\documentclass{amsart}
\usepackage{amssymb}
\usepackage{amsfonts}
\usepackage{geometry}

\newtheorem{theorem}{Theorem}
\theoremstyle{plain}

\newtheorem{corollary}[theorem]{Corollary}

\newtheorem{definition}[theorem]{Definition}

\newtheorem{lemma}[theorem]{Lemma}
\newtheorem{notation}[theorem]{Notation}

\newtheorem{proposition}[theorem]{Proposition}
\newtheorem{remark}[theorem]{Remark}

\numberwithin{equation}{section}
\input{tcilatex}
\begin{document}
	\title{Continuity of weak solutions to rough infinitely
		degenerate equations}
	\author{Lyudmila Korobenko}
	\address{Reed College\\
		Portland, Oregon, USA}
	\email{korobenko@reed.edu}
	\author{Cristian Rios}
	\address{University of Calgary\\
		Calgary, Alberta, Canada}
	\email{crios@ucalgary.ca}
	\author{Eric Sawyer}
	\address{McMaster University\\
		Hamilton, Ontario, Canada}
	\email{sawyer@mcmaster.ca}
	\author{Ruipeng Shen}
	\address{Center for Applied Mathematics\\
		Tianjin University \\
		Tianjin 300072, P.R.China}
	\email{srpgow@163.com}
	\date{\today }
	
	\begin{abstract}
		We obtain a generalization of the DeGiorgi Lemma to the infinitely
		degenerate regime and apply it to obtain continuity of weak
		solutions to certain infinitely degenerate equations. This reproduces the continuity result obtained in \cite{KoRiSaSh1} via Moser iteration, but only for homogeneous equations. However, the proofs are much less technical and more transparent.
	\end{abstract}
	
	\maketitle
	\tableofcontents

\section{Introduction}

In \cite{KoRiSaSh2}, building on work\ from \cite{KoRiSaSh1}, local
boundedness was established for weak subsolutions to certain infinitely
degenerate elliptic divergence form equations, motivated by the pioneering
work of Fedii \cite{Fe}, Kusuoka and Strook \cite{KuStr}, Morimoto \cite{Mor}
and Christ \cite{Chr}. The main theorem on local boundedness in \cite%
{KoRiSaSh2} included this.

\begin{theorem}[\protect\cite{KoRiSaSh2}]
	\label{Local}Suppose that $D\subset \mathbb{R}^{n}$ is a domain in $\mathbb{R%
	}^{n}$ with $n\geq 3$ and that 
	\begin{equation*}
	\mathcal{L}u\equiv \func{div}\mathcal{A}\left( x,u\right) \nabla u,\ \ \ \ \
	x=\left( x_{1},...,x_{n}\right) \in D,
	\end{equation*}%
	where$\ \mathcal{A}\left( x,z\right) \sim \left[ 
	\begin{array}{cc}
	I_{n-1} & 0 \\ 
	0 & f\left( x_{1}\right) ^{2}%
	\end{array}%
	\right] $, $I_{n-1}$ is the $\left( n-1\right) \times \left( n-1\right) $
	identity matrix, $\mathcal{A}$ has bounded measurable components, and the
	geometry $F=-\ln f$ satisfies the structure conditions in Definition \ref%
	{structure conditions} below.
	
	\begin{enumerate}
		\item If $F\leq D_{\sigma }$ for some $0<\sigma <1$, then every weak
		solution to $\mathcal{L}u=\phi $ with $A$-admissible $\phi $ is locally
		bounded in $D$.
		
		\item Conversely, if $n\geq 3$ and $\sigma >1$, then there exists an
		unbounded weak solution $u$ in a neighbourhood of the origin in $\mathbb{R}%
		^{n}$ to the equation $Lu=0$ with geometry $F=D_{\sigma }$.
	\end{enumerate}
\end{theorem}

Where geometry $D_{\sigma}$ is defined as $D_{\sigma}(x)\equiv\left( \frac{1}{\left\vert x\right\vert }%
\right) ^{\sigma }$, $x>0$.

The purpose of this paper is to improve the local boundedness conclusion in
part (1) of Theorem \ref{Local} to include continuity. 
For the geometric continuity theorem we need to consider a less degenerate
family of geometries. For $k\geq 0$ and $0<\sigma <\infty $, define $%
F_{k,\sigma }\left( r\right) =\left( \ln \frac{1}{r}\right) \left( \ln
^{\left( k\right) }\frac{1}{r}\right) ^{\sigma }$ and $f_{k,\sigma }\left(
r\right) =e^{-F_{k,\sigma }\left( r\right) }=r^{\left( \ln ^{\left( k\right)
	}\frac{1}{r}\right) ^{\sigma }}$. Note that $F_{0,\sigma }\left( r\right)
=\left( \ln \frac{1}{r}\right) \frac{1}{r^{\sigma }}$ and $F_{\sigma }\left(
r\right) =\frac{1}{r^{\sigma }}$ for $0<\sigma <\infty $ are essentially the
same families of geometries.

\begin{theorem}\label{Holder thm}
	Suppose that $\Omega \subset \mathbb{R}^{n}$ is a domain in $\mathbb{R}^{n}$
	with $n\geq 2$ and that 
	\begin{equation*}
	\mathcal{L}u\equiv \func{div}\mathcal{A}\left( x,u\right) \nabla u,\ \ \ \ \
	x=\left( x_{1},...,x_{n}\right) \in \Omega ,
	\end{equation*}%
	where$\ \mathcal{A}\left( x,z\right) \sim \left[ 
	\begin{array}{cc}
	I_{n-1} & 0 \\ 
	0 & f\left( x_{1}\right) ^{2}%
	\end{array}%
	\right] $, $I_{n-1}$ is the $\left( n-1\right) \times \left( n-1\right) $
	identity matrix, $\mathcal{A}$ has bounded measurable components, and the
	geometry $F=-\ln f$ satisfies the structure conditions in Definition \ref%
	{structure conditions}.
	
	\begin{enumerate}
		\item If $F\leq F_{3,\sigma }$ for some $0<\sigma <1$, then every weak
		solution to $\mathcal{L}u=0$ is continuous in $\Omega $.
		
		\item On the other hand, if $n\geq 3$ and $\sigma \geq 1$, then there exists
		a locally unbounded weak solution $u$ in a neighbourhood of the origin in $%
		\mathbb{R}^{n}$ to the equation $Lu=0$ with geometry $F=F_{0,\sigma }$.
	\end{enumerate}
\end{theorem}

\subsection{Preliminaries and definitions}

We recall some of the terminology and definitions from \cite{KoRiSaSh1} and 
\cite{KoRiSaSh2} that we use here. Let $A\left( x\right) $ be a nonnegative
semidefinite $n\times n$ matrix valued function in a bounded domain $\Omega
\subset \mathbb{R}^{n}$. We consider the second order \emph{special}
quasilinear equation (`special' because only $u$, and not $\nabla u$,
appears nonlinearly),%
\begin{equation*}
\mathcal{L}u\equiv \nabla ^{\limfunc{tr}}\mathcal{A}\left( x,u\left(
x\right) \right) \nabla u=\phi ,\ \ \ \ \ x\in \Omega ,
\end{equation*}%
and we assume the following quadratic form condition on the quasilinear
matrix $\mathcal{A}(x,u(x))$, 
\begin{equation}
k\,\xi ^{T}A(x)\xi \leq \xi ^{T}\mathcal{A}(x,z)\xi \leq K\,\xi ^{T}A(x)\xi
\ ,  \label{struc_0}
\end{equation}%
for a.e. $x\in \Omega $ and all $z\in \mathbb{R}$, $\xi \in \mathbb{R}^{n}$.
Here $k,K$ are positive constants and we assume that $A(x)=B\left( x\right)
^{\limfunc{tr}}B\left( x\right) $ where $B\left( x\right) $ is a Lipschitz
continuous $n\times n$ real-valued matrix defined for $x\in \Omega $. We
also consider the linear equation%
\begin{equation*}
Lu\equiv \nabla ^{\limfunc{tr}}A\left( x\right) \nabla u=\phi ,\ \ \ \ \
x\in \Omega ,
\end{equation*}%
and define the $A$-gradient by%
\begin{equation}
\nabla _{A}=B\left( x\right) \nabla \ .  \label{def A grad}
\end{equation}

\begin{definition}
	The degenerate Sobolev space $W_{A}^{1,2}\left( \Omega \right) $ is normed by%
	\begin{equation*}
	\left\Vert v\right\Vert _{W_{A}^{1,2}}\equiv \sqrt{\int_{\Omega }\left(
		\left\vert v\right\vert ^{2}+\nabla v^{\func{tr}}A\nabla v\right) }=\sqrt{%
		\int_{\Omega }\left( \left\vert v\right\vert ^{2}+\left\vert \nabla
		_{A}v\right\vert ^{2}\right) }.
	\end{equation*}
\end{definition}

\begin{definition}
	Let $\Omega $ be a bounded domain in $\mathbb{R}^{n}$. Assume that $\phi \in
	L_{\limfunc{loc}}^{2}\left( \Omega \right) $. We say that $u\in
	W_{A}^{1,2}\left( \Omega \right) $ is a \emph{weak solution} to $\mathcal{L}%
	u=\phi $ provided%
	\begin{equation*}
	-\int_{\Omega }\nabla w\left( x\right) ^{\limfunc{tr}}\mathcal{A}\left(
	x,u(x)\right) \nabla u=\int_{\Omega }\phi w
	\end{equation*}%
	for all $w\in \left( W_{A}^{1,2}\right) _{0}\left( \Omega \right) $, where $%
	\left( W_{A}^{1,2}\right) _{0}\left( \Omega \right) $ denotes the closure in 
	$W_{A}^{1,2}\left( \Omega \right) $ of the subspace of Lipschitz continuous
	functions with compact support in $\Omega $.
\end{definition}

Note that our quadratic form condition (\ref{struc_0}) implies that the
integral on the left above is absolutely convergent, and our assumption that 
$\phi \in L_{\limfunc{loc}}^{2}\left( \Omega \right) $ implies that the
integral on the right above is absolutely convergent. Weak sub and super
solutions are defined by replacing $=$ with $\geq $ and $\leq $ respectively
in the display above.

Given a geometry $F=-\ln f$, we define the balls $B$ to be the control balls
associated with the $n\times n$ matrix $M_{F}\left( x\right) =\left[ 
\begin{array}{cc}
I_{n-1} & 0 \\ 
0 & f\left( x_{1}\right) ^{2}%
\end{array}%
\right] $. Assuming the structure conditions in Definition \ref{structure
	conditions} below, we recall from \cite{KoRiSaSh1} that the Lebesgue measure
of the \emph{two} dimensional ball $B_{2D}\left( x,r\right) $ centered at $%
x\in \mathbb{R}^{2}$ with radius $r>0$ satisfies%
\begin{equation}
\left\vert B_{2D}\left( x,r\right) \right\vert \approx \left\{ 
\begin{array}{ccc}
r^{2}f(x_{1}) & \text{ if } & r\leq \frac{1}{\left\vert F^{\prime }\left(
	x_{1}\right) \right\vert } \\ 
\frac{f\left( x_{1}+r\right) }{\left\vert F^{\prime }\left( x_{1}+r\right)
	\right\vert ^{2}} & \text{ if } & r\geq \frac{1}{\left\vert F^{\prime
	}\left( x_{1}\right) \right\vert }%
\end{array}%
\right. .  \label{measure}
\end{equation}

\begin{definition}
	\label{def A admiss new}Let $\Omega $ be a bounded domain in $\mathbb{R}^{n}$
	and let $A\left( x\right) $ be a nonnegative semidefinite $n\times n$ matrix
	valued function as above. Fix $x\in \Omega $ and $\rho >0$. We say $\phi $
	is $A$\emph{-admissible} at $\left( x,\rho \right) $ if 
	\begin{equation*}
	\Vert \phi \Vert _{X\left( B\left( x,\rho \right) \right) }\equiv \sup_{v\in
		\left( W_{A}^{1,1}\right) _{0}(B\left( x,\rho \right) )}\frac{\int_{B\left(
			x,\rho \right) }\left\vert v\phi \right\vert \,dy}{\int_{B\left( x,\rho
			\right) }\Vert \nabla _{A}v\Vert \,dy}<\infty .
	\end{equation*}
\end{definition}

\begin{definition}[structure conditions]
	\label{structure conditions}We refer to the following five conditions on $%
	F:\left( 0,\infty \right) \rightarrow \mathbb{R}$ as \emph{structure
		conditions}:
	
	\begin{enumerate}
		\item $\lim_{x\rightarrow 0^{+}}F\left( x\right) =+\infty $;
		
		\item $F^{\prime }\left( x\right) <0$ and $F^{\prime \prime }\left( x\right)
		>0$ for all $x\in (0,R)$;
		
		\item $\frac{1}{C}\left\vert F^{\prime }\left( r\right) \right\vert \leq
		\left\vert F^{\prime }\left( x\right) \right\vert \leq C\left\vert F^{\prime
		}\left( r\right) \right\vert $ for $\frac{1}{2}r<x<2r<R$;
		
		\item $\frac{1}{-xF^{\prime }\left( x\right) }$ is increasing in the
		interval $\left( 0,R\right) $ and satisfies $\frac{1}{-xF^{\prime }\left(
			x\right) }\leq \frac{1}{\varepsilon }\,$for $x\in (0,R)$;
		
		\item $\frac{F^{\prime \prime }\left( x\right) }{-F^{\prime }\left( x\right) 
		}\approx \frac{1}{x}$ for $x\in (0,R)$.
	\end{enumerate}
\end{definition}

\begin{remark}
	We make no smoothness assumption on $f$ other than the existence of the
	second derivative $f^{\prime \prime }$ on the open interval $(0,R)$. Note
	also that at one extreme, $f$ can be of finite type, namely $f\left(
	x\right) =x^{\alpha }$ for any $\alpha >0$, and at the other extreme, $f$
	can be of strongly degenerate type, namely $f\left( x\right) =e^{-\frac{1}{%
			x^{\alpha }}}$ for any $\alpha >0$. Assumption (1) rules out the elliptic
	case $f\left( 0\right) >0$.
\end{remark}

\begin{notation}
	We refer to a function $F$ satisfying the structure conditions in Definition %
	\ref{structure conditions} as a `geometry' since $F=-\ln f$ then specifies
	the nonnegative semidefinite matrix $M_{F}=\left[ 
	\begin{array}{cc}
	I_{n-1} & 0 \\ 
	0 & f\left( x_{1}\right) ^{2}%
	\end{array}%
	\right] $ and hence the geometry of the associated control balls. The class
	of degenerate elliptic linear operators%
	\begin{equation*}
	Lu=\func{div}A\left( x\right) \nabla u,\ \ \ \ \ A\left( x\right) \sim
	M_{F}\left( x_{1}\right) ,
	\end{equation*}%
	is also specified along with the associated class of quasilinear operators%
	\begin{equation*}
	\mathcal{L}u=\func{div}\mathcal{A}\nabla u,\ \ \ \ \ \mathcal{A}\left(
	x,z\right) \sim M_{F}\left( x_{1}\right) .
	\end{equation*}
\end{notation}

\subsection{Control balls}

We now recall further notation from \cite{KoRiSaSh1} and \cite{KoRiSaSh2},
beginning with the case of $n=2$ dimensions. Let $d\left( x,y\right) $ be
the control metric on an open subset $\Omega $ of the plane $\mathbb{R}^{2}$
that is associated with the matrix $A$, and refer to the associated balls as
control balls, subunit balls, or $A$-balls. Now we recall the definition of
\textquotedblleft height\textquotedblright\ of an arbitrary $A$-ball. Let $%
X=(x_{1},0)$ be a point on the positive $x$-axis and let $r$ be a positive
real number. Let the upper half of the boundary of the ball $B(X,r)$ be
given as the graph of the function $\varphi \left( x\right) $, $%
x_{1}-r<x<x_{1}+r$. Denote by $\beta _{X,P}$ the geodesic that meets the
boundary of the ball $B(X,r)$ at the point $P=(x_{1}+r^{\ast },h)$ where $%
\beta _{X,P}$ has a vertical tangent at $P$, $r^{\ast }=r^{\ast }\left(
x_{1},r\right) $ and $h=h\left( x_{1},r\right) =\varphi \left( x_{1}+r^{\ast
}\right) $. Here both $r^{\ast }$ and $h$ are \emph{functions} of the two
independent variables $x_{1}$ and $r$, but we will often write $r^{\ast
}=r^{\ast }\left( x_{1},r\right) $ and $h=h\left( x_{1},r\right) $ for
convenience. We refer to $h=h\left( x_{1},r\right) $ as the \emph{height} of
the ball $B((x_{1},0),r)$. In \cite{KoRiSaSh1} the authors proved the
following estimates on the height.

\begin{proposition}
	\label{height}Let $\beta _{X,P}$, $r^{\ast }$ and $h$ be defined as above.
	Define $\lambda \left( x\right) $ implicitly by%
	\begin{equation*}
	r=\int_{x_{1}}^{x}\frac{\lambda \left( x\right) }{\sqrt{\lambda \left(
			x\right) ^{2}-f\left( u\right) ^{2}}}du.
	\end{equation*}%
	Then
	
	\begin{enumerate}
		\item For $x_{1}-r<x<x_{1}+r$ we have $\varphi \left( x\right) \leq \varphi
		\left( x_{1}+r^{\ast }\right) =h$.
		
		\item If $r\geq \frac{1}{\left\vert F^{\prime }\left( x_{1}\right)
			\right\vert }$, then 
		\begin{equation*}
		h\approx \frac{f\left( x_{1}+r\right) }{\left\vert F^{\prime }\left(
			x_{1}+r\right) \right\vert }\text{ and }r-r^{\ast }\approx \frac{1}{%
			\left\vert F^{\prime }\left( x_{1}+r\right) \right\vert }.
		\end{equation*}
		
		\item If $r\leq \frac{1}{\left\vert F^{\prime }\left( x_{1}\right)
			\right\vert }$, then%
		\begin{equation*}
		h\approx rf\left( x_{1}\right) \text{ and }r-r^{\ast }\approx r.
		\end{equation*}
	\end{enumerate}
\end{proposition}

Now consider a sequence of metric balls $\left\{ B\left( x,r_{k}\right)
\right\} _{k=1}^{\infty }$ centered at $x\in \Omega $ with radii $%
r_{k}\searrow 0$ such that $r_{0}=r$ and%
\begin{equation*}
\left\vert B\left( x,r_{k}\right) \setminus B\left( x,r_{k+1}\right)
\right\vert \approx \left\vert B\left( x,r_{k+1}\right) \right\vert ,\ \ \ \
\ k\geq 1,
\end{equation*}%
so that $B\left( x,r_{k}\right) $ is divided into two parts having
comparable area. We may in fact assume that%
\begin{equation}
r_{k+1}=\left\{ 
\begin{array}{lll}
r^{\ast }\left( x_{1},r_{k}\right) &  & \text{if }r_{k}\geq \frac{1}{%
	\left\vert F^{\prime }\left( x_{1}\right) \right\vert } \\ 
\frac{1}{2}r_{k} &  & \text{if }r_{k}<\frac{1}{\left\vert F^{\prime }\left(
	x_{1}\right) \right\vert }%
\end{array}%
\right.  \label{rkp1}
\end{equation}%
where $r^{\ast }$ is defined in Proposition \ref{height}. Now for $x_{1},t>0$
define 
\begin{equation*}
h^{\ast }\left( x_{1},t\right) =\int_{x_{1}}^{x_{1}+t}\frac{f^{2}\left(
	u\right) }{\sqrt{f^{2}\left( x_{1}+t\right) -f^{2}\left( u\right) }}du,
\end{equation*}%
so that $h^{\ast }\left( x_{1},t\right) $ describes the `height' above $%
x_{2} $ at which the geodesic through $x=\left( x_{1},x_{2}\right) $ curls
back toward the $y$-axis at the point $\left( x_{1}+t,x_{2}+h^{\ast }\left(
x_{1},t\right) \right) $. Then in the case $r_{k}\geq \frac{1}{\left\vert
	F^{\prime }\left( x_{1}\right) \right\vert }$, we have $h^{\ast }\left(
x_{1},r_{k+1}\right) =h\left( x_{1},r_{k}\right) $, $k\geq 0$, where $%
h\left( x_{1},r_{k}\right) $ is the height of $B\left( x,r_{k}\right) $. In
the opposite case $r_{k}<\frac{1}{\left\vert F^{\prime }\left( x_{1}\right)
	\right\vert }$, we have $r_{k+1}=\frac{1}{2}r_{k}$ instead, and we will
estimate differently.

For $k\geq 0$ define 
\begin{equation*}
E\left( x,r_{k}\right) \equiv \left\{ 
\begin{array}{ccc}
\left\{ y:x_{1}+r_{k+1}\leq y_{1}<x_{1}+r_{k},~\left\vert y_{2}\right\vert
<h^{\ast }\left( x_{1},y_{1}-x_{1}\right) \right\} & \text{ if } & r_{k}\geq 
\frac{1}{\left\vert F^{\prime }\left( x_{1}\right) \right\vert } \\ 
\left\{ y:x_{1}+r_{k+1}\leq y_{1}<x_{1}+r_{k},~\left\vert y_{2}\right\vert
<h^{\ast }\left( x_{1},r_{k}^{\ast }\right) =h\left( x_{1},r_{k}\right)
\right\} & \text{ if } & r_{k}<\frac{1}{\left\vert F^{\prime }\left(
	x_{1}\right) \right\vert }%
\end{array}%
\right. ,
\end{equation*}%
where we have written $r_{k}^{\ast }=r^{\ast }\left( x_{1},r_{k}\right) $
for convenience. In \cite{KoRiSaSh1} it was shown that 
\begin{equation}
\left\vert E\left( x,r_{k}\right) \right\vert \approx \left\vert E\left(
x,r_{k}\right) \bigcap B\left( x,r_{k}\right) \right\vert \approx \left\vert
B\left( x,r_{k}\right) \right\vert \text{ for all }k\geq 1,
\label{claim that}
\end{equation}%
and hence that 
\begin{equation*}
\left\vert E\left( x,r_{k}\right) \cap B\left( x,r_{k}\right) \right\vert
\geq \frac{1}{2}cr_{k}f\left( x_{1}\right) r_{k}\approx \left\vert B\left(
x,r_{k}\right) \right\vert \geq \left\vert E\left( x,r_{k}\right) \cap
B\left( x,r_{k}\right) \right\vert .
\end{equation*}

Now define $\Gamma \left( x,r\right) $ to be the set 
\begin{equation*}
\Gamma \left( x,r\right) =\dbigcup\limits_{k=1}^{\infty }E\left(
x,r_{k}\right) .
\end{equation*}%
The following lemma was proved in \cite{KoRiSaSh1}.

\begin{lemma}
	\label{lemma-subrepresentation}With notation as above, in particular with $%
	r_{0}=r$ and $r_{1}$ given by (\ref{rkp1}), and assuming $%
	\int_{E(x,r_{1})}w=0$, we have the subrepresentation formula%
	\begin{equation}
	w\left( x\right) \leq C\int_{\Gamma \left( x,r\right) }\left\vert \nabla
	_{A}w\left( y\right) \right\vert \frac{\widehat{d}\left( x,y\right) }{%
		\left\vert B\left( x,d\left( x,y\right) \right) \right\vert }dy,
	\label{subrepresentation}
	\end{equation}%
	where $\nabla _{A}$ is as in (\ref{def A grad}) and 
	\begin{equation*}
	\widehat{d}\left( x,y\right) \equiv \min \left\{ d\left( x,y\right) ,\frac{1%
	}{\left\vert F^{\prime }\left( x_{1}+d\left( x,y\right) \right) \right\vert }%
	\right\} .
	\end{equation*}
\end{lemma}

Note that when $f\left( r\right) =r^{N}$ is finite type, then $\widehat{d}%
\left( x,y\right) \approx d\left( x,y\right) $. Now define 
\begin{equation}
K_{r}\left( x,y\right) \equiv \frac{\widehat{d}\left( x,y\right) }{%
	\left\vert B\left( x,d\left( x,y\right) \right) \right\vert }\mathbf{1}%
_{\Gamma \left( x,r\right) }\left( y\right) ,  \label{kernel_est}
\end{equation}%
and for 
\begin{equation*}
y\in \Gamma \left( x,r\right) =\left\{ y\in B\left( x,r\right) :x_{1}\leq
y_{1}\leq x_{1}+r,\ \left\vert y_{2}-x_{2}\right\vert <h_{x,y}\right\} ,
\end{equation*}%
let $h_{x,y}=h^{\ast }\left( x_{1},y_{1}-x_{1}\right) $. Denote the dual
cone $\Gamma ^{\ast }\left( y,r\right) $ by%
\begin{equation*}
\Gamma ^{\ast }\left( y,r\right) \equiv \left\{ x\in B\left( y,r\right)
:y\in \Gamma \left( x,r\right) \right\} .
\end{equation*}%
Then we have%
\begin{eqnarray}
\Gamma ^{\ast }\left( y,r\right) &=&\left\{ x\in B\left( y,r\right)
:x_{1}\leq y_{1}\leq x_{1}+r,\ \left\vert y_{2}-x_{2}\right\vert
<h_{x,y}\right\}  \label{Gammastar} \\
&=&\left\{ x\in B\left( y,r\right) :y_{1}-r\leq x_{1}\leq y_{1},\ \left\vert
x_{2}-y_{2}\right\vert <h_{x,y}\right\} ,  \notag
\end{eqnarray}%
and consequently we get the `straight across' estimate,%
\begin{equation}
\int K_{r}\left( x,y\right) ~dx\approx \int_{y_{1}-r}^{y_{1}}\left\{
\int_{y_{2}-h_{x,y}}^{y_{2}+h_{x,y}}\frac{1}{h_{x,y}}dx_{2}\right\}
dx_{1}\approx \int_{x_{1}}^{x_{1}+r}dy_{1}=r\ .  \label{straight}
\end{equation}

\subsection{Higher dimensions}\label{higher_dim_sec}

Recall that in the two dimensional case, we have%
\begin{equation*}
\left\vert B_{2D}\left( x,d\left( x,y\right) \right) \right\vert \approx
h_{x,y}\widehat{d}\left( x,y\right) \approx h_{x,y}\min \left\{ d\left(
x,y\right) ,\frac{1}{\left\vert F^{\prime }\left( x_{1}+d\left( x,y\right)
	\right) \right\vert }\right\} .
\end{equation*}%
In the three dimensional case, the quantities $h_{x,y}$ and $\widehat{d}%
\left( x,y\right) $ remain formally the same (see Chapter 10 of \cite%
{KoRiSaSh1}) and we can write a typical geodesic in the form 
\begin{equation*}
\left\{ 
\begin{array}{l}
x_{2}=C_{2}\pm k\int_{0}^{x_{1}}\frac{\lambda }{\sqrt{\lambda ^{2}-[f(u)]^{2}%
}}\,du \\ 
x_{3}=C_{3}\pm \int_{0}^{x_{1}}\frac{[f(u)]^{2}}{\sqrt{\lambda
		^{2}-[f(u)]^{2}}}\,du%
\end{array}%
\right. ,
\end{equation*}%
so that a metric ball centered at $y=\left( y_{1},y_{2},y_{3}\right) $ with
radius $r>0$ is given by 
\begin{equation*}
B\left( y,r\right) \equiv \left\{ \left( x_{1},x_{2},x_{3}\right) :\left(
x_{1},x_{3}\right) \in B_{2D}\left( \left( y_{1},y_{3}\right) ,\sqrt{%
	r^{2}-\left\vert x_{2}-y_{2}\right\vert ^{2}}\right) \right\} ,
\end{equation*}%
where $B_{2D}\left( a,s\right) $ denotes the $2$-dimensional control ball
centered at $a$ in the plane parallel to the $x_{1},x_{3}$-plane with radius 
$s$ that was associated with $f$ above (see Corollaries 107 and 108 in \cite%
{KoRiSaSh1} and the subsequent paragraph).

In dimension $n\geq 4$, the same arguments show that a typical geodesic has
the form 
\begin{equation*}
\left\{ 
\begin{array}{l}
\mathbf{x}_{2}=\mathbf{C}_{2}\pm \mathbf{k}\int_{0}^{x_{1}}\frac{\lambda }{%
	\sqrt{\lambda ^{2}-[f(u)]^{2}}}\,du \\ 
x_{3}=C_{3}\pm \int_{0}^{x_{1}}\frac{[f(u)]^{2}}{\sqrt{\lambda
		^{2}-[f(u)]^{2}}}\,du%
\end{array}%
\right. ,
\end{equation*}%
where $\mathbf{x}_{2},\mathbf{C}_{2},\mathbf{k}\in \mathbb{R}^{n-2}$ are now 
$\left( n-2\right) $-dimensional vectors, so that a metric ball centered at 
\begin{equation*}
y=\left( y_{1},\mathbf{y}_{2},y_{3}\right) \in \mathbb{R}\times \mathbb{R}%
^{n-2}\times \mathbb{R}=\mathbb{R}^{n},
\end{equation*}%
with radius $r>0$ is given by 
\begin{equation*}
B\left( y,r\right) \equiv \left\{ \left( x_{1},\mathbf{x}_{2},x_{3}\right)
:\left( x_{1},x_{3}\right) \in B_{2D}\left( \left( y_{1},y_{3}\right) ,\sqrt{%
	r^{2}-\left\vert \mathbf{x}_{2}-\mathbf{y}_{2}\right\vert ^{2}}\right)
\right\} ,
\end{equation*}%
where $B_{2D}\left( a,s\right) $ denotes the $2$-dimensional control ball
centered at $a$ in the plane parallel to the $x_{1},x_{3}$-plane with radius 
$s$ that was associated with $f$ above. The following lemma was proved in 
\cite{KoRiSaSh2}, correcting Lemma 109 in Chapter 10 of \cite{KoRiSaSh1}.

\begin{lemma}
	\label{new}The Lebesgue measure of the \emph{three} dimensional ball $%
	B_{3D}\left( x,r\right) $ satisfies%
	\begin{equation*}
	\left\vert B_{3D}\left( x,r\right) \right\vert \approx \left\{ 
	\begin{array}{ccc}
	r^{3}f(x_{1}) & \text{ if } & r\leq \frac{2}{\left\vert F^{\prime }\left(
		x_{1}\right) \right\vert } \\ 
	\frac{f\left( x_{1}+r\right) }{\left\vert F^{\prime }\left( x_{1}+r\right)
		\right\vert ^{3}}\sqrt{r\left\vert F^{\prime }\left( x_{1}+r\right)
		\right\vert } & \text{ if } & r\geq \frac{2}{\left\vert F^{\prime }\left(
		x_{1}\right) \right\vert }%
	\end{array}%
	\right. ,
	\end{equation*}%
	and that of the $n$-dimensional ball $B_{nD}\left( x,r\right) $ satisfies%
	\begin{equation*}
	\left\vert B_{nD}\left( x,r\right) \right\vert \approx \left\{ 
	\begin{array}{ccc}
	r^{n}f(x_{1}) & \text{ if } & r\leq \frac{2}{\left\vert F^{\prime }\left(
		x_{1}\right) \right\vert } \\ 
	\frac{f\left( x_{1}+r\right) }{\left\vert F^{\prime }\left( x_{1}+r\right)
		\right\vert ^{n}}\left( r\left\vert F^{\prime }\left( x_{1}+r\right)
	\right\vert \right) ^{\frac{n}{2}-1} & \text{ if } & r\geq \frac{2}{%
		\left\vert F^{\prime }\left( x_{1}\right) \right\vert }%
	\end{array}%
	\right. .
	\end{equation*}
\end{lemma}

\section{Proportional vanishing $L^{1}$-Sobolev inequality}

Our geometric continuity theorem requires the proportional vanishing $L^{1}$%
-Sobolev inequality, which we will now establish. For simplicity we consider first the $2$-dimensional case.

Define 
\begin{equation}
K_{r}\left( x,y\right) \equiv \frac{\widehat{d}\left( x,y\right) }{%
	\left\vert B\left( x,d\left( x,y\right) \right) \right\vert }\mathbf{1}%
_{\Gamma \left( x,r\right) }\left( y\right) ,  \label{kernel_est}
\end{equation}%
and 
\begin{equation*}
\Gamma \left( x,r\right) =\left\{ y\in B\left( x,r\right) :x_{1}\leq
y_{1}\leq x_{1}+r,\ \left\vert y_{2}-x_{2}\right\vert <h^{\ast }\left(
x_{1},y_{1}-x_{1}\right) \right\} ,
\end{equation*}
and for $y\in \Gamma \left( x,r\right) $ let $h_{x,y}=h^{\ast }\left(
x_{1},y_{1}-x_{1}\right) $. Using the estimate $\left\vert B\left( x,d\left(
x,y\right) \right) \right\vert \approx h_{x,y}\widehat{d}(x,y)$ from Section \ref{higher_dim_sec} we have 
\begin{equation*}
K_{r}\left( x,y\right) \approx \frac{1}{h_{x,y}}\mathbf{1}_{\left\{ \left(
	x,y\right) :x_{1}\leq y_{1}\leq x_{1}+r,\ \left\vert y_{2}-x_{2}\right\vert
	<h_{x,y}\right\} }\left( x,y\right) .
\end{equation*}

Now denote the dual cone $\Gamma ^{\ast }\left( y,r\right) $ by%
\begin{equation*}
\Gamma ^{\ast }\left( y,r\right) \equiv \left\{ x\in B\left( y,r\right)
:y\in \Gamma \left( x,r\right) \right\} .
\end{equation*}%
Then we have%
\begin{eqnarray}
\Gamma ^{\ast }\left( y,r\right) &=&\left\{ x\in B\left( y,r\right)
:x_{1}\leq y_{1}\leq x_{1}+r,\ \left\vert y_{2}-x_{2}\right\vert
<h_{x,y}\right\}  \label{Gammastar} \\
&=&\left\{ x\in B\left( y,r\right) :y_{1}-r\leq x_{1}\leq y_{1},\ \left\vert
x_{2}-y_{2}\right\vert <h_{x,y}\right\} ,  \notag
\end{eqnarray}%
and consequently we get the `straight across' estimate in $n=2$ dimensions,%
\begin{equation}
\int K_{r}\left( x,y\right) ~dx\approx \int_{y_{1}-r}^{y_{1}}\left\{
\int_{y_{2}-h_{x,y}}^{y_{2}+h_{x,y}}\frac{1}{h_{x,y}}dx_{2}\right\}
dx_{1}\approx \int_{x_{1}}^{x_{1}+r}dy_{1}=r\ .  \label{straight}
\end{equation}

Turning now to the case of $n\geq 3$ dimensions, we have using Lemma \ref{new}
that%
\begin{equation*}
K_{B\left( 0,r_{0}\right) }\left( x,y\right) \approx 
\begin{cases}
\begin{split}
\frac{1}{r^{n-1}f(x_{1})}\mathbf{1}_{\widetilde{\Gamma }\left(
	x,r_{0}\right) }\left( y\right) ,\quad 0& <r=y_{1}-x_{1}<\frac{2}{|F^{\prime
	}(x_{1})|} \\
\frac{\left\vert F^{\prime }\left( x_{1}+r\right) \right\vert ^{n-1}}{%
	f(x_{1}+r)\lambda \left( x_{1},r\right) ^{n-2}}\mathbf{1}_{\widetilde{\Gamma 
	}\left( x,r_{0}\right) }\left( y\right) ,\quad R& \geq r=y_{1}-x_{1}\geq 
\frac{2}{|F^{\prime }(x_{1})|}
\end{split}%
\end{cases}%
,
\end{equation*}%
where $\lambda \left( x_{1},r\right) \equiv \sqrt{r\left\vert F^{\prime
	}\left( x_{1}+r\right) \right\vert }$. We denote the size of the kernel $%
K_{B\left( 0,r_{0}\right) }(x,y)$ as $\frac{1}{s_{y_{1}-x_{1}}}$ where 
\begin{equation*}
\frac{1}{s_{y_{1}-x_{1}}}\equiv 
\begin{cases}
\begin{split}
\frac{1}{r^{n-1}f\left( x_{1}\right) },\quad 0& <r=y_{1}-x_{1}<\frac{2}{%
	|F^{\prime }(x_{1})|} \\
\frac{|F^{\prime }(x_{1}+r)|^{n}}{f(x_{1}+r)\lambda \left( x_{1},r\right)
	^{n-2}},\quad 0& <r=y_{1}-x_{1}\geq \frac{2}{|F^{\prime }(x_{1})|}
\end{split}%
\end{cases}%
,
\end{equation*}%
and where the quantity $s_{r}$ can be, roughly speaking, thought of a cross
sectional volume analogous to the height $h_{r}$ in the two dimensional
case. We have%
\begin{eqnarray*}
	&&\int_{B_{+}\left( 0,r_{0}\right) }K_{B\left( 0,r_{0}\right) }\left(
	x,y\right) dy \\
	&=&\sum_{k=0}^{\infty }\int_{x_{1}+r_{k+1}}^{x_{1}+r_{k}}\left[
	\int_{\left\vert \mathbf{x}_{2}-\mathbf{y}_{2}\right\vert \leq \sqrt{%
			r_{k}^{2}-r_{k+1}^{2}}}\left\{ \int_{x_{3}-h^{\ast }\left(
		x_{1},r_{k}\right) }^{x_{3}+h^{\ast }\left( x_{1},r_{k}\right) }\frac{1}{%
		s_{y_{1}-x_{1}}}\left\vert B\left( 0,r_{0}\right) \right\vert dy_{3}\right\}
	d\mathbf{y}_{2}\right] dy_{1} \\
	&=&\sum_{k=0}^{\infty }\int_{x_{1}+r_{k+1}}^{x_{1}+r_{k}}\left[ \left( \sqrt{%
		r_{k}^{2}-r_{k+1}^{2}}\right) ^{n-2}2h^{\ast }\left( x_{1},r_{k}\right) %
	\right] \frac{1}{s_{y_{1}-x_{1}}}dy_{1} \\
	&\approx &\sum_{k=0}^{\infty
	}\int_{x_{1}+r_{k+1}}^{x_{1}+r_{k}}s_{y_{1}-x_{1}}\frac{1}{s_{y_{1}-x_{1}}}%
	dy_{1}=\int_{x_{1}}^{x_{1}+r_{0}}dy_{1}=r_{0},
\end{eqnarray*}%
where the approximation in the fourth line above comes from the estimates%
\begin{eqnarray*}
	\left( \sqrt{r_{k}^{2}-r_{k+1}^{2}}\right) ^{n-2}2h^{\ast }\left(
	x_{1},r_{k}\right) \left( r_{k}-r_{k+1}\right) &=&\left\vert \widetilde{E}%
	\left( x,r_{k}\right) \right\vert \approx \left\vert B\left( x,r_{k}\right)
	\right\vert \approx s_{r_{k}}\left( r_{k}-r_{k+1}\right) , \\
	s_{r_{k}} &\approx &s_{y_{1}-x_{1}},\ \ \ \ \ \text{for }x_{1}+r_{k+1}\leq
	y_{1}<x_{1}+r_{k}\ .
\end{eqnarray*}%
This gives the $n$-dimensional `straight across' estimate, 
\begin{equation}
\int_{B\left( 0,r_{0}\right) }K_{B\left( 0,r_{0}\right) }\left( x,y\right)
dy\approx r_{0}\ .  \label{straight n}
\end{equation}

We can now prove the proportional vanishing $L^{1}$-Sobolev inequality by appealing to the the $\left( 1,1\right) $ \emph{Poincar\'{e}}
inequality in \cite{KoRiSaSh2}. We recall it here for convenience
\begin{proposition}
	\label{1 1 Poin'} Let the balls $B(0,r)$ and the degenerate gradient $\nabla
	_{A}$ be as above. There exists a constant $C$ such that the Poincar\'{e}
	Inequality 
	\begin{equation*}
	\int_{B(0,r)}\left\vert w(x)-\bar{w}\right\vert dx\leq
	Cr\int_{B(0,2r)}|\nabla _{A}w|dx
	\end{equation*}%
	holds for any Lipschitz function $w$ and sufficiently small $r>0$. Here $%
	\bar{w}$ is the average defined by 
	\begin{equation*}
	\bar{w}=\frac{1}{|B(0,r)|}\int_{B(0,r)}wdx.
	\end{equation*}
\end{proposition}

This leads to the following proportional vanishing $L^{1}$-Sobolev inequality
\begin{proposition}
	\label{vanishing}Let the balls $B(0,r)$ and the degenerate gradient $\nabla
	_{A}$ be as above. There exists a constant $C$ such that the proportional
	vanishing $L^{1}$-Sobolev inequality 
	\begin{equation}
	\int_{B(0,r)}\left\vert w\right\vert dx\leq Cr\int_{B(0,2r)}|\nabla _{A}w|dx,
	\label{proportional_sob}
	\end{equation}%
	holds for any Lipschitz function $w$ that vanishes on a subset $E$ of the
	ball $B\left( 0,r\right) $ with $\left\vert E\right\vert \geq \frac{1}{2}%
	\left\vert B\left( 0,r\right) \right\vert $, and all sufficiently small $r>0$%
	.
\end{proposition}

\begin{proof}
	We have 
	\begin{align}
	\int_{B(0,r)}\left\vert w\right\vert dx& =\int_{B(0,r)}\left\vert w(x)-\frac{%
		1}{|E\cap B|}\int_{E\cap B}w(y)dy\right\vert dx  \notag  \label{poin_est_1}
	\\
	& \leq \frac{1}{|E\cap B|}\int \int_{B\times E\cap B}\left\vert
	w(x)-w(y)\right\vert dxdy\lesssim \frac{1}{|B|}\int \int_{B\times
		B}\left\vert w(x)-w(y)\right\vert dxdy.
	\end{align}%
	Next, it follows from the proof of Propositions \ref{1 1 Poin'} that 
	\begin{equation}
	\frac{1}{|B|}\int_{B\times B}\left\vert w(x)-w(y)\right\vert dxdy\leq
	Cr\int_{2B}\left\vert \nabla _{A}w\right\vert .  \label{poin_est_2}
	\end{equation}%
	Estimate \ref{proportional_sob} follows from (\ref{poin_est_1}) and (\ref%
	{poin_est_2}).
\end{proof}

\subsection{Orlicz-Sobolev inequality}
In this section we state the Orlicz-Sobolev inequality proved in \cite{KoRiSaSh2}
\begin{equation}
\left\Vert w\right\Vert _{L^{\Phi }\left( \mu _{r_{0}}\right) }\leq C\varphi
\left( r_{0}\right) \ \left\Vert \nabla _{A}w\right\Vert _{L^{1}\left( \mu
	_{r_{0}}\right) },\ \ \ \ \ w\in \left( W_{A}^{1,1}\right) _{0}\left(
B\left( 0,r_{0}\right) \right)  \label{Phi N norm}
\end{equation}%
where the particular family of Young functions $\Phi$ we are interested in, is defined as follows
\begin{equation}
\Phi _{N}\left( t\right) \equiv \left\{ 
\begin{array}{ccc}
t(\ln t)^{N} & \text{ if } & t\geq E=E_{N}=e^{2N} \\ 
\left( \ln E\right) ^{N}t & \text{ if } & 0\leq t\leq E=E_{N}=e^{2N}%
\end{array}%
\right. .  \label{def Phi N ext}
\end{equation}%
This is Proposition 70 in \cite{KoRiSaSh2}
\begin{proposition}
	\label{sob_nd} Let $n\geq 2$. Assume that for some $C>0$ the function 
	\begin{equation}
	\varphi (r)\equiv C|F^{\prime }\left( r\right) |^{N}r^{N+1}
	\label{phi_incr N}
	\end{equation}%
	satisfies $\lim_{r\rightarrow 0}\varphi (r)=0$. Assume in addition that
	geometry $F$ satisfies 
	\begin{equation}
	F^{\prime \prime }(r)\leq \left( 1+\frac{1-\varepsilon }{N}\right) \frac{%
		|F^{\prime }(r)|}{r},\quad r\in (0,r_{0}),\quad \varepsilon >0.
	\label{geom_extra_cond}
	\end{equation}%
	Then:
	
	\begin{enumerate}
		\item the $\left( \Phi ,\varphi \right) $-Sobolev inequality (\ref{Phi N
			norm}) holds with geometry $F$, with $\varphi $ as in (\ref{phi_incr N}),
		and with $\Phi $ as in (\ref{def Phi N ext}), $N>1$,
		
		\item and if $\varphi _{\max }\left( r\right) \equiv \sup_{0<s<r_{0}}\varphi
		(s)<\infty $ is a finite constant function, then the $\left( \Phi ,\varphi
		_{\max }\right) $-Sobolev inequality (\ref{Phi N norm}) holds with
		geometry $F$, with $\varphi $ as in (\ref{phi_incr N}), and with $\Phi $ as
		in (\ref{def Phi N ext}), $N>1$,
		
		\item in particular, if for some $\varepsilon >0$ we have 
		\begin{equation}
		\left\vert F^{\prime }\left( r\right) \right\vert \leq C\left( \frac{1}{r}%
		\right) ^{1+\frac{1-\varepsilon }{N}},  \label{F_prime_bound N}
		\end{equation}%
		then the $\left( \Phi ,\varphi _{\max }\right) $-Sobolev inequality (\ref%
		{Phi N norm}) holds with geometry $F$ and $\varphi _{\max }(r)\equiv C$.
	\end{enumerate}
\end{proposition}

\subsection{The DeGiorgi Lemma}

Here is an infinitely degenerate variation on the DeGiorgi Lemma in Lemma
1.4 of Caffarelli and Vasseur \cite{CaVa}, but yielding an estimate
different from that of Caffarelli and Vasseur - one that does not involve an
isoperimetric inequality. For convenience we recall the
proportional vanishing $L^{1}$-Sobolev inequality (\ref{proportional_sob}) from the previous section:%
\begin{eqnarray}
&&\ \ \ \ \ \ \ \ \ \ \int_{B}\left\vert w\right\vert \leq Cr\left( B\right)
\int_{2B}\left\vert \nabla _{A}w\right\vert ,  \label{prop van'} \\
&&\text{for all Lipschitz }w\text{ supported in }2B\text{ that vanish on a
	subset }E\text{ of a ball }B\text{ with }\left\vert E\right\vert \geq \frac{1%
}{2}\left\vert B\right\vert .  \notag
\end{eqnarray}

\begin{lemma}
	\label{DeG Lemma}Suppose that the proportional vanishing $L^{1}$-Sobolev
	inequality (\ref{prop van'}) holds. Fix $x$ and $r$ and suppose that $w$
	satisfies $\int_{B\left( x,2r\right) }\left\vert \nabla _{A}w_{+}\left(
	y\right) \right\vert ^{2}dy\leq C_{0}$. Set 
	\begin{eqnarray*}
		\mathcal{A} &\equiv &\left\{ y\in B\left( x,r\right) :w\left( y\right) \leq
		0\right\} , \\
		\mathcal{C} &\equiv &\left\{ z\in B\left( x,r\right) :w\left( z\right) \geq
		1\right\} , \\
		\mathcal{D} &\equiv &\left\{ y\in B\left( x,2r\right) :0<w\left( y\right)
		<1\right\} .
	\end{eqnarray*}%
	Then if $\left\vert \mathcal{A}\right\vert \geq \frac{1}{2}\left\vert
	B\left( x,r\right) \right\vert $, we have 
	\begin{equation}
	C_{0}\left\vert \mathcal{D}\right\vert \geq C_{1}\left( \frac{\left\vert 
		\mathcal{A}\right\vert \left\vert \mathcal{C}\right\vert }{r\left\vert
		B\left( x,r\right) \right\vert }\right) ^{2}.  \label{isoperimetric}
	\end{equation}
\end{lemma}

\begin{proof}
	Let $\overline{w}\left( y\right) \equiv \max \left\{ 0,\min \left\{
	1,w\left( y\right) \right\} \right\} $, and note that $\overline{w}\left(
	z\right) =1$ for $z\in \mathcal{C}$. Then applying (\ref{prop van'}) with $w=%
	\overline{w}$, $B=B\left( x,r_{0}\right) $ and $E=\mathcal{A}$, we have that%
	\begin{eqnarray*}
		\left\vert \mathcal{C}\right\vert \left\vert \mathcal{A}\right\vert &=&\int_{%
			\mathcal{C}}\overline{w}\left( z\right) dz\ \left\vert \mathcal{A}%
		\right\vert \leq \int_{B\left( x,r\right) }\overline{w}\left( z\right) dz\
		\left\vert \mathcal{A}\right\vert \\
		&\leq &Cr\left\vert B\left( x,r\right) \right\vert \int_{B\left( x,2r\right)
		}\left\vert \nabla _{A}\overline{w}\left( y\right) \right\vert
		dy=Cr\left\vert B\left( x,r\right) \right\vert \int_{\mathcal{D}}\left\vert
		\nabla _{A}w\left( y\right) \right\vert dy \\
		&\lesssim &r_{0}\left\vert B\left( x,r\right) \right\vert \sqrt{\left\vert 
			\mathcal{D}\right\vert }\left\Vert \nabla _{A}w\right\Vert _{L^{2}\left(
			B\left( x,2r\right) \right) }\lesssim \sqrt{C_{0}}r\left\vert B\left(
		x,r\right) \right\vert \sqrt{\left\vert \mathcal{D}\right\vert }.
	\end{eqnarray*}%
	Thus we obtain%
	\begin{equation*}
	\left\vert \mathcal{C}\right\vert \left\vert \mathcal{A}\right\vert \lesssim
	r\left\vert B\left( x,r\right) \right\vert \sqrt{C_{0}\left\vert \mathcal{D}%
		\right\vert },
	\end{equation*}%
	or%
	\begin{equation*}
	C_{0}\left\vert \mathcal{D}\right\vert \geq C_{1}\left( \frac{\left\vert 
		\mathcal{A}\right\vert \left\vert \mathcal{C}\right\vert }{r\left\vert
		B\left( x,r\right) \right\vert }\right) ^{2}.
	\end{equation*}
\end{proof}

\section{Continuity of locally bounded weak solutions}

\subsection{Local boundedness}
We first recall Corollary 23 from \cite{KoRiSaSh2} to the local boundedness result that will be used in the proof of continuity theorem.

\begin{corollary}
	\label{renorm}Suppose all the assumptions of Theorem \ref{Local} are
	satisfied. Then 
	\begin{equation}
	\left\Vert u_{+}\right\Vert _{L^{\infty }(\frac{1}{2}B)}\leq A_{N,\varepsilon}\left(
3r\right) \left( \frac{1}{\left\vert 3B\right\vert }\int_{B}u_{+}^{2}\right)
^{\frac{1}{2}}+\left\Vert \phi \right\Vert _{X}\ ,  \label{Inner ball inequ'}
	\end{equation}%
	\begin{equation}
	\text{where }A_{N,\varepsilon}\left( r\right) =C_{1}\exp \left(
	C_{2}\left( \frac{\varphi \left( r\right) }{r}\right) ^{\frac{1}{N-1-\varepsilon}}\right) .  \label{def AN}
	\end{equation}
\end{corollary}

\subsection{Caccioppoli inequality}
Recall from \cite[see (3.2)]{KoRiSaSh2} that if $u$ is a weak subsolution to 
$\mathcal{L}u=0$, then we have the standard Cacciopoli inequality 
\begin{equation}\label{standard Cacc}
\int_{B}|\nabla _{A}(\psi u_{+})|^{2}\leq C\left( ||\psi ||_{L^{\infty
}}+||\nabla _{A}\psi ||_{L^{\infty }}\right) ^{2}\int_{B\cup supp\psi
}u_{+}^{2}.
\end{equation}%
We repeat the proof here, which simplifies in the homogeneous setting. From
the assumption that $\nabla A\nabla u=0$ in the weak sense we obtain%
\begin{eqnarray*}
	0 &=&\int \left( \psi ^{2}u_{+}\right) \nabla A\nabla u=-\int \nabla \left(
	\psi ^{2}u_{+}\right) A\nabla u_{+} \\
	&=&-\int 2\psi u_{+}\left( \nabla \psi \right) A\nabla u_{+}-\int \psi
	^{2}\left( \nabla u_{+}\right) A\nabla u_{+}\ ,
\end{eqnarray*}%
which gives%
\begin{equation*}
\int \psi ^{2}\left\vert \nabla _{A}u_{+}\right\vert ^{2}=-\int 2\psi
u_{+}\left( \nabla \psi \right) A\nabla u_{+}\leq 2\left( \int \left\vert
\psi \nabla _{A}u_{+}\right\vert ^{2}\right) ^{\frac{1}{2}}\left( \int
\left\vert \nabla _{A}\psi \right\vert ^{2}u_{+}^{2}\right) ^{\frac{1}{2}},
\end{equation*}%
hence%
\begin{equation*}
\int \left\vert \psi \nabla _{A}u_{+}\right\vert ^{2}\leq 4\int \left\vert
\nabla _{A}\psi \right\vert ^{2}u_{+}^{2}\ .
\end{equation*}%
We now use%
\begin{equation*}
\int_{B}|\nabla _{A}(\psi u_{+})|^{2}=\int_{B}|u_{+}\nabla _{A}\psi +\psi
\nabla _{A}u_{+}|^{2}\leq 2\left\{ \int_{B}|u_{+}\nabla _{A}\psi
|^{2}+\int_{B}|\psi \nabla _{A}u_{+}|^{2}\right\}
\end{equation*}%
to complete the proof of Caccioppoli. 

\subsection{Proof of Theorem \ref{Holder thm}}
We can now prove Theorem \ref{Holder thm} for weak solutions to a
homogeneous degenerate equation. In fact it is easily seen, using Corollary %
\ref{renorm}, that it suffices to prove the following local statement with $%
\frac{1}{2\sqrt{\delta \left( r\right) }}=A_{N,\varepsilon}(3r)$, where $A_{N,\varepsilon}(r)$ is
the constant in the local boundedness inequality (\ref{Inner ball inequ'}) defined in (\ref{def AN}).

\begin{proposition}
	\label{Prop}Let $B_{r}=B\left( x,r\right) $. Suppose that (\ref%
	{isoperimetric}) holds, and also that for some $\delta \left( r\right) >0$,
	we have the following local boundedness inequality,%
	\begin{equation}
	\left\Vert u_{+}\right\Vert _{L^{\infty }\left( B_{\frac{r}{2}}\right) }\leq 
	\frac{1}{2\sqrt{\delta \left( r\right) }}\left( \frac{1}{\left\vert
		B_{3r}\right\vert }\int_{B_{r}}u_{+}^{2}\right) ^{\frac{1}{2}}\ ,\ \ \ \ \ 
	\text{whenever }Lu=0\text{ in }B_{r}\text{ ,}  \label{LB'}
	\end{equation}%
	for all $0<r<r_{0}$. Moreover we assume the summability condition%
	\begin{equation}
	\sum_{j=1}^{\infty }\lambda _{j}=\infty ,  \label{summ cond}
	\end{equation}%
	where $r_{j}=\frac{r_{0}}{4^{j}}$ and $\lambda _{j}\equiv \frac{1}{2^{3+%
			\frac{C_{3}}{\delta ^{2}(r_{j})}}}$ for $j\geq 1$, and where $C_{3}$ is a
	positive constant. Then if $u$ is a weak solution to $\mathcal{L}u=0$ in $%
	B_{r_{0}}$, we conclude that $u$ is continuous at $x$.
\end{proposition}

We now reduce the proof of Proposition \ref{Prop} to proving Proposition \ref%
{osc} below, which shows that if a solution $v$ is bounded by $1$ in a ball $%
B_{3r}$, and is nonpositive on a `sufficiently large' subset of the smaller
ball $B_{r}$, then $v$ is in fact bounded by a constant \emph{less} than $1$
in the smaller ball $B_{r}$. This reduction is achieved in two steps by way
of the following lemma.

\begin{lemma}
	\label{oscill}Suppose that the local boundedness property in Proposition \ref%
	{Prop} holds, equivalently (\ref{LB'}). Let $u$ be a weak solution of $%
	\mathcal{L}u=0$ in $B_{r_{0}}$. Then with $\lambda \left( r\right) \equiv 
	\frac{1}{2^{3+\frac{C_{3}}{\delta (r)^{2}}}}$ we have%
	\begin{equation*}
	\limfunc{osc}_{B_{\frac{r}{2}}}u\leq \left( 1-\frac{\lambda \left( r\right) 
	}{2}\right) \limfunc{osc}_{B_{r}}u,\ \ \ \ \ 0<r<r_{0}.
	\end{equation*}
\end{lemma}

Indeed, iterating Lemma \ref{oscill} gives%
\begin{equation*}
\limfunc{osc}_{B_{\frac{r}{4^{\ell }}}}u\leq \left\{
\dprod\limits_{j=1}^{\ell }\left( 1-\frac{\lambda \left( \frac{r}{4^{j}}%
	\right) }{2}\right) \right\} \limfunc{osc}_{B_{r}}u,\ \ \ \ \ \ell \geq 1,
\end{equation*}%
which implies $\lim_{\ell \rightarrow \infty }\limfunc{osc}_{B_{\frac{r}{%
			4^{\ell }}}}u=0$ since the infinite product above vanishes if the
summability condition (\ref{summ cond}) holds. Hence $u$ is continuous at $x$%
, which proves Proposition \ref{Prop}.

Next we note that Lemma \ref{oscill} is in turn an easy consequence of this
proposition.

\begin{proposition}
	\label{osc}Suppose that the local boundedness property (\ref{LB'}) holds.
	With $r$ sufficiently small, let%
	\begin{equation*}
	\lambda \left( r\right) \equiv \frac{1}{2^{3+\frac{C_{3}}{\delta(r)^{2}}}}%
	\in \left( 0,1\right) .
	\end{equation*}%
	Let $v\leq 1$ and $\mathcal{L}v=0$ in $B_{3r}$. Assume that $\left\vert
	B_{r}\cap \left\{ v\leq 0\right\} \right\vert \geq \frac{1}{2}\left\vert
	B_{r}\right\vert $. Then 
	\begin{equation*}
	\sup_{B_{\frac{r}{2}}}v\leq 1-\lambda \left( r\right) .
	\end{equation*}
\end{proposition}

Indeed, to prove Lemma \ref{oscill} from Proposition \ref{osc}, let%
\begin{equation*}
v\left( x\right) \equiv \frac{2}{\limfunc{osc}_{B_{r}}u}\left\{ u\left(
x\right) -\frac{\sup_{B_{r}}u+\inf_{B_{r}}u}{2}\right\} ,
\end{equation*}%
so that $-1\leq v\leq 1$ on $B_{r}$. If $\left\vert \left\{ v\leq 0\right\}
\cap B_{r}\right\vert \geq \frac{1}{2}\left\vert B_{r}\right\vert $, then
Proposition \ref{osc} gives $\limfunc{osc}_{B_{\frac{r}{2}}}v\leq 2-\lambda
\left( r\right) $, which implies%
\begin{equation*}
\limfunc{osc}_{B_{\frac{r}{2}}}u\leq \left( 1-\frac{\lambda \left( r\right) 
}{2}\right) \limfunc{osc}_{B_{r}}u.
\end{equation*}%
If instead we have $\left\vert \left\{ v\geq 0\right\} \cap B_{r}\right\vert
\geq \frac{1}{2}\left\vert B_{r}\right\vert $, we obtain the same inequality
by applying the above argument to $-v$. This completes the proof\ of Lemma %
\ref{oscill}.

Thus matters have been reduced to proving Proposition \ref{osc}.

\begin{proof}[Proof of Proposition \protect\ref{osc}]
	Define 
	\begin{equation*}
	w_{k}=2^{k}\left( v-(1-2^{-k})\right)
	\end{equation*}%
	and let $\varphi \in C_{0}^{\infty }(B_{3r})$ be such that $\varphi =1$ on $%
	B_{2r}$ and $||\nabla _{A}\varphi ||_{L^{\infty }(B_{3r})}\leq \frac{2}{r}$.
	Using $\mathcal{L}w_{k}=0$ in $B_{3r_{0}}$ and the standard Caccioppoli
	inequality (\ref{standard Cacc}) with $\phi =0$ we have 
	\begin{equation*}
	\int_{B_{2r}}|\nabla _{A}(w_{k})_{+}|^{2}\leq \int_{B_{3r}}|\nabla
	_{A}(\varphi (w_{k})_{+})|^{2}\leq C||\nabla _{A}\varphi ||_{L^{\infty
	}}^{2}\int_{B_{3r}\cap supp\varphi }(w_{k})_{+}^{2}\leq 4C\left\vert
	B_{3r}\right\vert r^{-2}\ ,
	\end{equation*}%
	where the last inequality follows from the fact that $w_{k}\leq 1$, which in
	turn follows from the assumption that $v\leq 1$. Define%
	\begin{eqnarray*}
		\mathcal{A}_{k} &=&\left\{ 2w_{k}\leq 0\right\} \cap B_{r}\ , \\
		\mathcal{C}_{k} &=&\left\{ 2w_{k}\geq 1\right\} \cap B_{r}\ , \\
		\mathcal{D}_{k} &=&\left\{ 0<2w_{k}<1\right\} \cap B_{2r}\ .
	\end{eqnarray*}%
	Also note that $v\leq 0$ implies $w_{k}\leq 0$, and therefore we have 
	\begin{equation}
	\left\vert \mathcal{A}_{k}\right\vert =\left\vert \left\{ 2w_{k}\leq
	0\right\} \cap B_{r}\right\vert =\left\vert \left\{ w_{k}\leq 0\right\} \cap
	B_{r}\right\vert \geq \frac{1}{2}\left\vert B_{r}\right\vert ,  \label{A}
	\end{equation}%
	so that the hypothesis $\left\vert \mathcal{A}_{k}\right\vert \geq \frac{1}{2%
	}\left\vert B_{r}\right\vert $ in the proportional vanishing $L^{1}$-Sobolev
	inequality (\ref{prop van'}) holds.
	
	We will apply Lemma \ref{DeG Lemma} with $w=$ $2w_{k}$ recursively for $%
	k=0,1,2,...$ below, but only as long as 
	\begin{equation}
	\int_{B_{r}}(w_{k+1})_{+}^{2}\frac{dx}{|B_{3r}|}\geq \delta \left( r\right) ,
	\label{C'}
	\end{equation}%
	where $\delta \left( r\right) $ is the positive constant in Proposition \ref%
	{Prop}. We use both (\ref{A}) and (\ref{C'}), and the fact that%
	\begin{equation*}
	w_{k+1}=2w_{k}-1
	\end{equation*}%
	implies $\mathcal{C}_{k}=\left\{ w_{k+1}\geq 0\right\} \cap B_{r}$, and
	hence 
	\begin{equation*}
	\left\vert \mathcal{C}_{k}\right\vert \geq \int_{\mathcal{C}%
		_{k}}(w_{k+1})_{+}^{2}dx=\int_{B_{r}}(w_{k+1})_{+}^{2}dx\geq \delta \left(
	r\right) |B_{3r}|,
	\end{equation*}%
	to obtain from Lemma \ref{DeG Lemma} that%
	\begin{eqnarray*}
		\left\vert \left\{ 0<w_{k}<\frac{1}{2}\right\} \cap B_{2r}\right\vert
		&=&\left\vert \left\{ 0<2w_{k}<1\right\} \cap B_{2r}\right\vert =\left\vert 
		\mathcal{D}_{k}\right\vert \\
		&\geq &\frac{C_{1}}{C\left\vert B_{3r}\right\vert r^{-2}}\left( \frac{\frac{1%
			}{2}\left\vert B_{r}\right\vert \delta \left( r\right) |B_{3r}|}{r\left\vert
			B_{r}\right\vert }\right) ^{2}=\frac{C_{1}}{4C}\delta \left( r\right)
		^{2}\left\vert B_{3r}\right\vert =\alpha \left( r\right) ,
	\end{eqnarray*}%
	where 
	\begin{equation}
	\alpha \left( r\right) \equiv \frac{C_{1}}{4C}\delta \left( r\right)
	^{2}\left\vert B_{3r}\right\vert >0  \label{def alpha}
	\end{equation}%
	depends on $r$, but not on $k$ or $v$. This gives 
	\begin{align*}
	\left\vert B_{2r}\right\vert & \geq \left\vert \left\{ w_{k}\leq 0\right\}
	\cap B_{2r}\right\vert =|\{2w_{k-1}\leq 1\}\cap B_{2r}| \\
	& =\left\vert \left\{ w_{k-1}\leq 0\right\} \cap B_{2r}\right\vert
	+\left\vert \left\{ 0<w_{k-1}\leq \frac{1}{2}\right\} \cap B_{2r}\right\vert
	\\
	& \geq \left\vert \left\{ w_{k-1}\leq 0\right\} \cap B_{2r}\right\vert
	+\alpha \\
	& \vdots \\
	& \geq \left\vert \left\{ w_{0}\leq 0\right\} \cap B_{2r}\right\vert
	+k\alpha \left( r\right) \geq k\alpha \left( r\right) ,
	\end{align*}
	
	The above inequality, namely $\left\vert B_{2r}\right\vert \geq k\alpha
	\left( r\right) $, must fail for a finite $k$ \emph{independent} of $v$, in
	fact it fails for the unique integer $k_{0}\geq 0$ satisfying $\frac{%
		\left\vert B_{2r}\right\vert }{\alpha \left( r\right) }<k_{0}\leq \frac{%
		\left\vert B_{2r}\right\vert }{\alpha \left( r\right) }+1$, and for this $%
	k_{0}$ we must then have 
	\begin{equation*}
	\int_{B_{r}}(w_{k_{0}+1})_{+}^{2}dx<\delta \left( r\right) \left\vert
	B_{3r}\right\vert .
	\end{equation*}
	
	By the local boundedness inequality (\ref{LB'}), we conclude that in the
	ball $B_{\frac{r}{2}}$, we have 
	\begin{equation*}
	w_{k_{0}+1}\leq \left\Vert \left( w_{k_{0}+1}\right) _{+}\right\Vert
	_{L^{\infty }\left( B_{\frac{r}{2}}\right) }\leq \frac{1}{2\sqrt{\delta
			\left( r\right) }}\left( \frac{1}{\left\vert B_{3r}\right\vert }%
	\int_{B_{r}}\left( w_{k_{0}+1}\right) _{+}^{2}\right) ^{\frac{1}{2}}<\frac{1%
	}{2\sqrt{\delta \left( r\right) }}\sqrt{\delta \left( r\right) }=\frac{1}{2}.
	\end{equation*}%
	Rescaling back to $v$ now gives%
	\begin{eqnarray*}
		&&2^{k_{0}+1}\left[ v-\left( 1-\frac{1}{2^{k_{0}+1}}\right) \right]
		_{+}=\left( w_{k_{0}+1}\right) _{+}\leq \frac{1}{2}\text{ in }B_{\frac{r}{2}}
		\\
		\Longrightarrow \text{ \ \ \ \ } &&v\leq 1-\frac{1}{2^{k_{0}+1}}+\frac{1}{2}%
		\frac{1}{2^{k_{0}+1}}=1-\frac{1}{2^{k_{0}+2}}\text{ in }B_{\frac{r}{2}}.
	\end{eqnarray*}%
	Finally, we note that (\ref{def alpha}) gives $\alpha \left( r\right) =\frac{%
		C_{1}}{4C}\delta \left( r\right) ^{2}\left\vert B_{3r}\right\vert $, and
	hence 
	\begin{equation*}
	k_{0}\leq \frac{\left\vert B_{2r}\right\vert }{\alpha \left( r\right) }+1=%
	\frac{\left\vert B_{2r}\right\vert }{\frac{C_{1}}{4C}\delta \left( r\right)
		^{2}\left\vert B_{3r}\right\vert }+1\leq \frac{C_{3}}{\delta \left( r\right)
		^{2}}+1,
	\end{equation*}%
	and thus that 
	\begin{equation*}
	\frac{1}{2^{k_{0}+2}}\geq \frac{1}{2^{3+\frac{C_{3}}{\delta \left( r\right)
				^{2}}}}=\lambda (r).
	\end{equation*}
\end{proof}

We can now obtain our main result,
Theorem \ref{Holder thm}, directly from Part (1) of Theorem \ref{Local} and Proposition \ref{Prop}. To check the summability condition \ref{summ cond} we first note that for the geometry $F_{\sigma ,k}$,
i.e. $F\left( r\right) =F_{\sigma ,k}=\left( \ln \frac{1}{r}\right) \left(
\ln ^{(k)}\frac{1}{r}\right) ^{\sigma }$, we have 
\begin{equation*}
\left\vert F^{\prime }\left( r\right) \right\vert \approx \frac{1}{r}\left(
\ln ^{(k)}\frac{1}{r}\right) ^{\sigma },\quad F^{\prime \prime }\left(
r\right) \leq (1+\varepsilon )\frac{|F^{\prime }\left( r\right) |}{r},
\end{equation*}%
where $\varepsilon $ can be made arbitrarily small by choosing $r_{0}$
sufficiently small. Thus the conditions of Proposition (\ref{sob_nd}) are
satisfied and we have the following estimate for the superradius: 
\begin{equation*}
\varphi \left( r\right) \leq Cr\left( \ln ^{(k)}\frac{1}{r}\right) ^{\sigma
	N}.
\end{equation*}%
Now we use the formula (\ref{def AN}),%
\begin{equation*}
A_{N,\varepsilon}\left( r\right) \equiv C_{1}\exp \left\{ C_{2}\left( \frac{\varphi
	\left( r\right) }{r}\right) ^{\frac{1}{N-1-\varepsilon}}\right\} ,
\end{equation*}%
together with $A_{N,\varepsilon}\left( 3r\right) =\frac{1}{2\sqrt{%
		\delta \left(r\right) }}$, to obtain that for
the geometry $F=F_{k,\sigma }$ we have%
\begin{equation*}
\delta \left( r\right) =\frac{1}{4A_{N,\varepsilon}\left( 3r\right) ^{2}}=\frac{1}{%
	4C_{1}^{2}}e^{-2C_{2}\left( \frac{\varphi \left( 3r\right) }{3r}\right) ^{%
		\frac{1}{N-1-\varepsilon}}}\geq e^{-C^{\prime }\left( \ln ^{(k)}\frac{1}{r}\right)
	^{\frac{\sigma N}{N-1-\varepsilon} }}.
\end{equation*}
Note that if $\sigma <1$ we can find $N>1$ satisfying
\[
N>\frac{1+\varepsilon}{1-\sigma},
\]
which implies
\[
\gamma\equiv \frac{\sigma N}{N-1-\varepsilon}<1.
\]
Now choose $r_{j}=4^{-j}r_{0}$ and let $k=3$, we thus have 
\begin{equation*}
\frac{1}{\delta \left( r_{j}\right) ^{2}}\leq e^{C^{\prime \prime }\left[
	\ln ^{\left( 3\right) }\left( \frac{1}{r_{j}}\right) \right] ^{\gamma }}.
\end{equation*}%
Now we use $\lambda _{j}\equiv \frac{1}{2^{3+\frac{C_{3}}{\delta ^{2}(r_{j})}%
}}$ from (\ref{summ cond}), together with the inequality%
\begin{eqnarray*}
	&&\left[ \ln ^{\left( k\right) }a\right] ^{\gamma }\leq \varepsilon \ln
	^{\left( k\right) }a,\ \ \ \ \ \text{for }a\text{ sufficiently large} \\
	&&\text{\ \ \ \ \ \ \ \ \ \ \ \ \ \ \ depending on }\gamma <1\text{ and }%
	\varepsilon >0,
\end{eqnarray*}%
to obtain 
\begin{eqnarray*}
	\sum_{j=1}^{\infty }\lambda _{j} &=&\sum_{j=1}^{\infty }\frac{1}{2^{3+\frac{%
				C_{3}}{\delta ^{2}(r_{j})}}}\approx \sum_{j=1}^{\infty }\frac{1}{%
		2^{3+C_{3}e^{C^{\prime \prime }\left[ \ln ^{\left( 3\right) }\left( \frac{1}{%
					r_{j}}\right) \right] ^{\gamma }}}} \\
	&\gtrsim &\sum_{j=1}^{\infty }2^{-C^{\prime \prime \prime }\varepsilon \ln
		^{\left( 2\right) }\left( \frac{4^{j}}{r_{0}}\right) }\gtrsim
	\sum_{j=1}^{\infty }e^{-C^{\prime \prime \prime \prime }\varepsilon \ln
		j}=\sum_{j=1}^{\infty }\frac{1}{j^{C^{\prime \prime \prime \prime
			}\varepsilon }}=\infty ,
\end{eqnarray*}%
if $\varepsilon >0$ is chosen sufficiently small. This establishes the summability condition and finishes the proof of Theorem \ref{Holder thm}.

\end{document}